\documentclass{amsart}

\usepackage{amssymb, mathabx, bigints}
\usepackage{amsfonts}
\usepackage{amsmath}
\usepackage{amsthm}
\usepackage{mathtools}
\usepackage{graphicx}
\usepackage{mathrsfs}
\usepackage{dsfont}
\usepackage{amscd}
\usepackage{multirow}
\usepackage[all]{xy}
\normalfont
\usepackage[T1]{fontenc}
\usepackage{calligra}
\usepackage{verbatim}
\usepackage[usenames,dvipsnames]{color}
\usepackage[colorlinks=true,linkcolor=Blue,citecolor=Violet]{hyperref}
\usepackage{enumerate}
\usepackage{slashed}
\usepackage{nicematrix}
\usepackage{faktor}

\newcommand{\N}{{\mathds{N}}}
\newcommand{\Z}{{\mathds{Z}}}

\newcommand{\R}{{\mathds{R}}}
\newcommand{\C}{{\mathds{C}}}
\newcommand{\T}{{\mathds{T}}}

\newcommand{\A}{{\mathfrak{A}}}

\newcommand{\bigslant}[2]{{\raisebox{.2em}{$#1$}\left/\raisebox{-.2em}{$#2$}\right.}}

\newcommand{\Lip}[1][L]{{\mathsf{#1}}}

\newcommand{\Hilbert}{{\mathscr{H}}}

\newcommand{\spectralpropinquity}[1]{{\mathsf{\Lambda}^{\mathsf{spec}}_{#1}}}

\newcommand{\Kantorovich}[1]{{\mathsf{mk}_{#1}}}

\newcommand{\StateSpace}{{\mathscr{S}}}

\newcommand{\qcms}{quantum compact metric space}

\newcommand{\sa}[1]{{\mathfrak{sa}\left({#1}\right)}}

\newcommand{\dom}[1]{{\operatorname*{dom}\left({#1}\right)}}

\newcommand{\norm}[2]{\left\|{#1}\right\|_{#2}}

\newcommand{\worknote}[1]{}
\newcommand{\opnorm}[3]{{\left|\mkern-1.5mu\left|\mkern-1.5mu\left| {#1} \right|\mkern-1.5mu\right|\mkern-1.5mu\right|_{#3}^{#2}}}

\newcommand{\alg}[1]{{\mathfrak{#1}}}

\newcommand{\nctwosolenoid}{\mathcal{A}^{2,p}_{\theta}}

\theoremstyle{plain}
\newtheorem{theorem}{Theorem}[section]

\newtheorem{lemma}[theorem]{Lemma}
\newtheorem{proposition}[theorem]{Proposition}

\newtheorem{theorem-definition}[theorem]{Theorem-Definition}

\theoremstyle{definition}
\newtheorem{definition}[theorem]{Definition}

\theoremstyle{remark}

\newtheorem*{acknowledgement*}{Acknowledgement}

\renewcommand{\geq}{\geqslant}
\renewcommand{\leq}{\leqslant}

\newcommand{\Dirac}{{\slashed{D}}}

\newcommand{\solenoid}{{\mathfrak s}}

\numberwithin{equation}{section}

\allowdisplaybreaks[4]

\begin{document}
\title[]{Spectral Triples on noncommutative solenoids from the standard spectral triples on quantum tori}
\author{Carla Farsi}
\email{carla.farsi@colorado.edu}
\address{Department of Mathematics \\ University of Colorado at Boulder \\ Boulder CO 80309-0395}

\author{Fr\'{e}d\'{e}ric Latr\'{e}moli\`{e}re}
\email{frederic@math.du.edu}
\urladdr{http://www.math.du.edu/\symbol{126}frederic}
\address{Department of Mathematics \\ University of Denver \\ Denver CO 80208}

\author{Judith Packer}
\email{judith.jesudason@colorado.edu}
\address{Department of Mathematics \\ University of Colorado at Boulder \\ Boulder CO 80309-0395}

\date{\today}

\begin{abstract}
We address the natural question: as noncommutative solenoids are inductive limits of quantum tori, do the standard spectral triples on quantum tori converge to some spectral triple on noncommutative solenoid for the spectral propinquity? We answer this question by showing that, using appropriate bounded perturbations of the spectral triples on quantum tori, such a spectral triple on noncommutative solenoid can be constructed.
\end{abstract}
\maketitle


\section{Introduction}

A. Connes initiated \cite{Connes81,Connes,Connes89} a far-reaching program to adapt and extend classical tools from topology and Riemannian geometry to operator--algebraic noncommutative settings.  
As part of this program, Connes introduced the notion of a \textit{spectral triple} \cite{Connes89,Connes} which  in its most basic form, is given by  a unital $C^*$-algebra $\A$ represented on a Hilbert space $\Hilbert$, and a self-adjoint operator $\Dirac$ with compact resolvent, which boundedly commutes with a dense $\ast$-subalgebra $\A_{\Dirac}$ of $\A$. The fundamental example of such a spectral triple is given by the Dirac operator on a spin closed manifold. In particular, the operator $\Dirac$ is called the \emph{Dirac operator} of the spectral triple $(\A,\Hilbert,\Dirac)$.

Connes observed that a spectral triple $(\A,\Hilbert,\Dirac)$ induces an (extended, pseudo) metric on the state space $\StateSpace(\A)$ of $\A$ by setting, for any two states $\varphi,\psi \in \StateSpace(\A)$:
\begin{equation}\label{Connes-dist-eq}
	\Kantorovich{\Dirac}(\varphi,\psi) \coloneqq \sup\left\{ |\varphi(a) - \psi(a)| : a\in\A_{\Dirac}, \opnorm{[\Dirac,a]}{}{\Hilbert} \leq 1 \right\}
\end{equation}
where $\opnorm{T}{}{\Hilbert}$ is the operator norm of any bounded operator $T$. When this construction leads to an actual metric which metrizes
the weak$^*$-topology of $\StateSpace(\A)$, the spectral triple $(\A,\Hilbert,\Dirac)$ is called \emph{metric}.

\medskip

A metric spectral triple $(\A,\Hilbert,\Dirac)$ gives rise to a noncommutative analogue of the algebra of Lipschitz functions over a compact metric space, as studied by Rieffel \cite{Rieffel99, Rieffel00, Rieffel21} under the name of a {\qcms}.  Rieffel's motivation for the study of these structures is the development of a noncommutative analogue \cite{Rieffel00} of the Gromov-Hausdorff distance \cite{Gromov, Gromov81}, which in turn, enabled the formalization of various convergences found in mathematical physics, where for instance, a sphere is seen as a limit of matrix algebras. 

The definition of {\qcms s} evolved in time. We will work with the definition used by the second author, which is designed to enable the construction of an analogue of the Gromov-Hausdorff distance called the \emph{propinquity}, which presents some advantages over Rieffel's original quantum Gromov-Hausdorff distance (at the cost of more involved proofs at times). The main advantages of the propinquity,  which induces the Gromov-Hausdorff topology on the class of compact metric spaces, are two-fold. First,  it is a complete metric that  can only be zero between {\qcms s} built on  *-isomorphic C*-algebras. Second, it is designed so that it can be extended to structures beyond {\qcms s}, such as modules over {\qcms s}, monoid actions over {\qcms s}, and, putting these ideas together, to metric spectral triples \cite{Latremoliere13b, Latremoliere13c, Latremoliere15b,Latremoliere16b,  Latremoliere13, Latremoliere15, Latremoliere14, Latremoliere16c, Latremoliere18c,  Latremoliere18a,Latremoliere18b,Latremoliere21a,Latremoliere18d,Latremoliere18g,Latremoliere22}.

\medskip

The \emph{spectral propinquity} is a distance on the class of metric spectral triples, up to unitary equivalence between spectral triples \cite{Latremoliere21a, Latremoliere18g, Latremoliere22, Rieffel21}. It enables the discussion of approximations of spectral triples by others, in a geometric sense. Such natural properties, as the spectrum and the functional calculus of the Dirac operators in spectral triples, are continuous for the spectral propinquity; see \cite{Latremoliere22} for details on this. Convergence of spectral triples, as well as convergence of spectral triples truncations, has become of subject of interest as of late; see e.g. \cite{Glazier2020,  CvS, Leimbach24, Rieffel21}.

\medskip

In this paper, we apply the theory of convergence of spectral triples to an interesting family of noncommutative spaces, called the \emph{noncommutative solenoids}. Noncommutative $d$-solenoids were introduced and studied  by Latr\'emoli\`ere and  Packer in \cite{Latremoliere16,Latremoliere11c}. Let us fix, here and throughout this paper, a natural number $p \in \N\setminus\{0,1\}$; we denote the group of $p$-adic rationals by:
\begin{equation*}
	\Z\left[\frac{1}{p}\right] \coloneqq \bigcup_{n\in\N} p^{-n}\Z = \left\{ \frac{m}{p^n} : m\in\Z, n \in \N \right\} \text.
\end{equation*}
There are two natural topologies on the group $\Z\left[\frac{1}{p}\right]$: the one inhertired as a subset of $\R$, and the discrete topology --- we will fix the latter and return to the role of the former later on. The dual group of $\Z\left[\frac{1}{p}\right]$ is the solenoid group:
\begin{equation*}
	\solenoid_p \coloneqq \left\{ (z_n)_{n\in\N} \in \T^\N : \forall n \in \N \quad z_{n+1}^p = z_n \right\} \text.
\end{equation*}
A \emph{noncommutative $d$-solenoid} ($d\in \N\setminus\{0,1\}$) is the twisted group C*-algebra
\begin{equation*}
	\alg{S}_\sigma \coloneqq C^\ast\left(\Z\left[\frac{1}{p}\right]^d, \sigma\right)
\end{equation*}
for any $\T$-valued $2$-cocycle of $\Z\left[\frac{1}{p}\right]^d$. We refer to \cite{Latremoliere11c} for a discussion of these cocycles, and the classification of these algebras in terms of their cocycles in the case $d=2$. In general, we will assume that our $2$-cocycles are normalized, i.e. $\sigma(g,-g) = 1$ for all $g \in \Z\left[\frac{1}{p}\right]^d$.

Since $\Z\left[\frac{1}{p}\right]^p = \bigcup_{n\in\N} p^{-n}\Z^d$, and since the restriction of a $2$-cocycle to a subgroup is again a $2$-cocycle, the C*-algebra $\alg{S}_\sigma$ is an indutive limit of quantum tori, i.e. twisted convolution C*-algebra of $\Z^d$. Noncommutative tori are prototypes of quantum spaces, and are endowed with a natural spectral triple: in fact, noncommutative tori naturally act on $L^2(\T^d)$, which then makes it possible to define spectral triples over them by using a standard version of the Dirac operator over tori. As  noncommutative solenoids are inductive limits of noncommutative tori, it thus becomes very natural to ask what, if any, is the limit of the spectral triples over noncommutative tori, as a possible spectral triple over noncommutative solenoids. This problem is the focus of the present paper.

\medskip

Let us contextualize this problem. Since early in the study of the propinquity, the question of the relationship betweem convergence in the categorical sense (inductive limits) and in the metric sense (for the propinquity) has received attention. For instance, the second author and K. Aguilar started the study of convergence for AF algebras \cite{Latremoliere15}, and K. Aguilar continued this research with various interesting observations. In \cite{FarsiLatremolierePacker22}, a comprehensive study of the relationship between inductive limits of spectral triples \cite{Floricel19} and the propinquity was applied, in particular, to the construction of metric spectral triples on twisted convolution C*-algebras of groups obtained as inductive limits of subgroups. Examples include certain Bunce-Deddens algebras and noncommutative solenoids.

Now, the spectral triples on noncommutative solenoids in \cite{FarsiLatremolierePacker22} were not the first ones introduced --- not even the first quantum metric over these algebras. Indeed, in \cite{Latremoliere16}, the Latr\'emoli\`ere and Packer introduced the first quantum metric on noncommutative solenoids which made them limits, in the propinquity, of quantum tori. However in that case, the convergence was purely metric (no spectral triple and no spectral propinquity, which had yet to be discovered). Later on, and closer to our current work, the first and third author, together with Landry and Larsen \cite{FarsiPacker22}, used the work of Connes \cite{Connes89}, Rieffel and Christ \cite{Rieffel15b} and Long and Wu \cite{Long-Wu1, LongWu21} to construct a metric spectral tripe on noncommutative tori using a well chosen length function over the groups $\Z\left[\frac{1}{p}\right]^d$. At this point, it became natural to ask whether this spectral triple was a limit of spectral triples on quantum tori. To address this question, as seen in \cite{FarsiLatremolierePacker22}, a refinement of the spectral triple in \cite{FarsiPacker22} was used. The process described in \cite{FarsiLatremolierePacker22} is quite general.

However, the spectral triples on quantum tori constructed in \cite{FarsiLatremolierePacker22} which converge to the metric spectral triple over noncommutative solenoids also built in \cite{FarsiLatremolierePacker22}, are not quite the usual ones over quantum tori. There are two changes made from the standard spectral triples over quantum tori: they are bounded perturbations (first change) of the absolute value (second change) of the usual spectral triples over quantum tori.

Therefore, as stated above, in this paper, we wish to see if we can refine the contruction in \cite{FarsiLatremolierePacker22} to obtain, for the special case of noncommutative solenoids, some result about the convergence of the standard spectral triples. At this point, we make two observations which will drive this work. First, the need for some form of bounded perturbation in the same vein as \cite{FarsiLatremolierePacker22} will again appear, as it seems necessary (the explanation being the same as the need for the same perturbation in \cite{FarsiLatremolierePacker22}). This lead us to our main results, Theorems \eqref{thm:triple is metric} and  \eqref{thm:sp  GH lim}. The proofs here reveal some nice properties of these spectral triples which take advantage of the particular nature of noncommutative solenoids (which are built using discrete subgroups of $\R^d$, allowing for this particular construction).

Second, one key reason why we are working with the standard spectral triple over quantum tori, rather than their absolute value, is that these spectral triples have a nontrivial K-homology class --- which, we note in passing, would not be changed by a bounded perturbation. So, in essence, we are interested in the behavior of K-homology classes, as well. This seems important enough --- and natural enough --- a question to address here.

We note in passing that various classes of spectral triples over quantum tori have already been studied from a metric perspective (see \cite{Latremoliere05, Latremoliere15c, Latremoliere21a}). However, here, we explicitely ask for the convergence of spectral triples to a limit over a noncommutative solenoids rather than a quantum torus, so this work is quite distinct.

We also illustrate the discontinuity of K-homology with respect to the spectral propinquity. This is expected, of course, since for instance, all compact metric spaces can be seen as metric limits of finite metric spaces, the latter being homologically trivial --- but not geometrically so. Here, we observe that moreover, K-homology does not behave well with respect to inductive limits either (for example, for the Bunde-Deddens algebra, its $K^0$--homology group is zero \cite{Klimek23}). So, while our spectral triple over noncommutative solenoids is the natural limit of non-trivial (K-homologically speaking) spetral triples over quantum tori, we prove that in fact, there is \emph{no}  metric spectral triple on noncommutative $2$-solenoids with nonzero index map  (see Proposition \eqref{prp:no-K-hom-class}). This is an interesting phenomenon: from our geometric perspective, we isolate a very interesting spectral triple over noncommutative solenoids which gives it a  natural geometry as a limit of quantum tori, with the added bonus that the restriction of our spectral triple to the canonical quantum tori inside the noncommutative solenoids are the ``natural'' spectral triples there (up to a bounded perturbation), and our construction is natural and well-behaved --- for instance, the bounded functional calculi for the spectral triples over quantum tori converge to the bounded functional calculus for our spectral triple here, which has consequences about the continuity of the spectral actions (see \cite{Latremoliere22}) --- if one were to do physics over a noncommutative solenoid, this would be the natural path. Yet K-homology is blind to this construction.


Our paper is naturally structured in two sections: we first introduce our new metric spectral triples and prove that they are limits of spectral triples on quantum tori. We then show an obstruction result which precludes metric spectral triples on noncommutative $2$-solenoids to have nontrivial K-homology class.

In this paper we will freely use the framework and notation of \cite{Latremoliere21a, Latremoliere18g,Latremoliere22, Rieffel21}
(and especially \cite{FarsiLatremolierePacker22})
to which we refer the reader for details on the propinquity and the spectral propinquity.

\begin{acknowledgement*} The authors thank A. Gorokhovsky for helpful conversations. 
This work was partially supported by the Simons Foundation (Simons Foundation collaboration grant \#523991 [C. Farsi] and \# 31698 [J. Packer].)
\end{acknowledgement*}

\section{A Metric Spectral Triple on Noncommutative Solenoids}

We fix $d, p \in \N\setminus\{0,1\}$. We will work with the following groups:
\begin{equation*}
	G_\infty \coloneqq \Z\left[\frac{1}{p}\right]^d = \bigcup_{n\in\N} G_n \text{ where }G_n \coloneqq \left(p^{-n} \Z\right)^d \text{ for all $n\in\N$.}
\end{equation*}
For each $j \in \{1,\ldots,d\}$, we denote the $j$-th canonical surjection from $G_\infty$ by  $x_j : (z_1,\ldots,z_d) \in G_\infty \mapsto z_j \text.$

We now define a function $F$ on $G_\infty$, akin to a $p$-adic absolute value, to keep track of which subgroup of the form $G_n$ an element naturally belongs to; formally:
\begin{equation*}
	\mathds{F} : (z_1,\ldots,z_d) \in G_\infty \mapsto p^{\min\{ n \in \N : (p^n z_1,\ldots,p^n z_d) \in \Z^d \} } \text.
\end{equation*}

These functions allow the definition of a natural length function over the group $G_\infty$, by setting
\begin{equation}\label{length-eq}
	\mathds{L} : z \in G_\infty \longmapsto \sqrt{\sum_{j=1}^d |x_j(z)|^2 + \mathds{F}(z)^2} \text.
\end{equation}
By \cite[Corollary 4.17]{FarsiLatremolierePacker22}, the function $\mathds{L}$ is a proper length function with the bounded doubling property, in the following sense:
\begin{definition}
	A function $l : G\rightarrow[0,\infty)$ is a length function over a group $G$ when:
	\begin{enumerate}
		\item $l(g) = 0$ if and only if $g = 1$,
		\item $l(g) = l(g^{-1})$ for all $g \in G$,
		\item $l(g h) \leq l(g) + l(h)$ for all $g,h \in G$.
		
	\end{enumerate}
	If $G_l[r] \coloneqq \{ g \in G : l(g) \leq r \}$ is finite for all $r > 0$, then we say that $l$ is \emph{proper}.
\end{definition}

\begin{definition}
	A proper length function $l$ on a discrete group $G$ has \emph{bounded-doubling} when there exists $C > 0$ such that, for all $r \geq 1$,
	\begin{equation*}
	|G_{l}[2r]| \leq C\ |G_{l}[r]| \text,
	\end{equation*}
	where we denoted by $|\cdot|$ the cardinality of the given sets.
\end{definition}

\medskip

We define the following subspace:
\begin{equation*}
	H \coloneqq \left\{ f \in L^2\left(G_\infty\right) :
	\sum_{z \in G_\infty} \mathds{L}(z)^2 |f(z)|^2 < \infty \right\} \text.
\end{equation*}
On $H$, for each $j\in\{1,\ldots,d\}$, we define the multiplication operator $X_j$ by
\begin{equation*}
	X_j : f \in H \mapsto x_j f \text.
\end{equation*}
as well as the multiplication operator:
\begin{equation*}
	M_{\mathds{F}} : f \in H \mapsto \mathds{F} f \text.
\end{equation*}

We assemble our Dirac operator from the above ingredients, by introducing $d+1$ anticommuting self-adjoint unitaries $\gamma_1$,\ldots,$\gamma_{d+1}$ on some Hermitian space $E$, i.e. the images of the canonical generators of the Clifford algebra of $\C^{d+1}$ by some chosen representation on $E$. We then define on $\dom{\Dirac} \coloneqq H\otimes E$:
\begin{equation}\label{Our-Dirac-eq}
	\Dirac_\infty \coloneqq \sum_{j=1}^d X_j \otimes \gamma_j + M_{\mathds{F}}\otimes\gamma_{d+1} \text.
\end{equation}
Henceforth, we identify $\ell^2(G_\infty)\otimes E$ with $\ell^2(G_\infty,E)$.
\medskip

We now fix a $\T$-valued $2$-cocycle $\sigma$ of $G_\infty$. Of course, $G_\infty$ acts on $\ell^2\left(G_\infty\right)$ via its left regular $\sigma$-projective representation. Specifically, for all $g \in G_\infty$, and for all $\xi \in \ell^2\left(G_\infty\right)$, we set
\begin{equation*}
	\lambda_\infty(g) \xi : h \in G_\infty \longmapsto \sigma_g^h\xi(h-g) \text{ where }\sigma_g^h \coloneqq \sigma(g,h-g) \text.
\end{equation*}
Thus defined, $\lambda_\infty$ is a unitary $\sigma$-projective representation of $G_\infty$. More generally, for any $f \in C_c\left(G_\infty\right)$, we define $\lambda_\infty(f) \coloneqq \sum_{g \in G_\infty} f(g)\lambda_\infty(g)$. The C*-algebra $C^\ast\left(G_\infty,\sigma\right)$ is identified with the closure of $\lambda_\infty\left(C_c\left(G_\infty\right)\right)$ for the operator norm (since $G_\infty$ is Abelian, hence amenable).

We propose to study in this section the spectral triple
\begin{equation}\label{Our-spectral-triple-eq}
	\left( C^\ast\left(G_\infty,\sigma \right),\ell^2\left(G_\infty\right)\otimes E, \Dirac_\infty \right) \text.
\end{equation}

Of keen interest to us are the restrictions of our spectral triple to quantum tori. Specifically, we identify $\ell^2(G_n)\otimes E$ with the subspace \[\left\{ \xi \in \ell^2(G_\infty)\otimes E : \forall g \in G_\infty\setminus G_n \quad \xi(g) = 0 \right\}\text.\] It is a standard exercise to check that $\lambda_\infty(f) \ell^2(G_n)\otimes E \subseteq \ell^2(G_n)\otimes E$ whenever $f \in C_c(G_n)$, identified again in the obvious way as a subspace (in fact, a subalgebra) of $C_c(G_\infty)$. 

An easy exercise shows that $\Dirac_\infty \dom{\Dirac_\infty}\cap \ell^2(G_n)\otimes E\subseteq \ell^2(G_n)\otimes E$. We define the restriction of $\Dirac_\infty$ to $\dom{\Dirac_n} \coloneqq \dom{\Dirac}\cap \ell^2(G_n)\otimes E$ by $\Dirac_n$. We thus obtain spectral triples on quantum tori of the form:
\begin{equation}\label{Our-spectral-triple-n-eq}
	\left(C^\ast(G_n,\sigma), \ell^2(G_n)\otimes E, \Dirac_n\right)\text,
\end{equation}
where we denote the restriction of $\sigma$ to $G_n$ again by $\sigma$.

The representation $\lambda_\infty$ restricted to $C^\ast(G_n,\sigma)$ leaves the subspace $\ell^2(G_n,E)$ invariant, and we call the induced representation $\lambda_n$; with our identifications, $\lambda_n$ is of course the left regular $\sigma$-projective representation of the quantum torus $C^\ast(G_n,\sigma)$. \emph{We now observe that the spectral triple given in Equation \eqref{Our-spectral-triple-n-eq} is a bounded perturbation of the usual spectral triple over the quantum torus $C^\ast(G_n,E)$ up to a scaling.}

\medskip

A first step in our work is of course to check that we indeed have defined spectral triples.
\begin{lemma}
	$(C^\ast(G_n,\sigma),\ell^2(G_n)\otimes E,\Dirac_n)$ is a spectral triple for all $n\in \N\cup\{\infty\}$.
\end{lemma}

\begin{proof}
	Fix $n\in \N\cup\{\infty\}$. Let $g \in G_n$, and note that $\lambda_n(g) \dom{\Dirac_n}\subseteq\dom{\Dirac_n}$. Let $\xi \in \ell^2(G_n,E)$. For all $h\in G_n$, we compute:
	\begin{align*}
	[\Dirac_n,\lambda_n(g)] \xi(h)
	&= \Dirac_n\lambda_n(g)\xi(h) - \lambda_n(g)\Dirac_n\xi(h) \\
	&= \left(\sum_{j=1}^d x_j(h)\gamma_j \sigma_{g}^h \xi(h-g)  + \mathds{F}(h) \sigma_{g}^h \gamma_{d+1}\xi(h-g)\right) \\
	&- \left( \sum_{j=1}^d x_j(h-g) \sigma_{g}^h \gamma_j\xi(h-g) + \mathds{F}(h-g)\sigma_{g}^h \gamma_{d+1}\xi(h-g)  \right)  \\
	&= \left(\sum_{j=1}^d x_j(g) \gamma_j + (\mathds{F}(h)-\mathds{F}(h-g)) \gamma_{d+1}\right)\sigma_{g}^h  \xi(h-g) .
	\end{align*}
	
	A quick computation shows that $|\mathds{F}(h) -\mathds{F}(h-g)| \leq \mathds{F}(g)$. Therefore, $[\Dirac_n,\lambda_n(g)]$ is bounded. Consequently, if $f \in C_c(G_\infty)$ then $\lambda_n(f)\dom{\Dirac_n}\subseteq\dom{\Dirac_n}$, and $[\Dirac_n,\lambda_n(f)]$ is bounded.
	
	It is a standard argument that $\Dirac_n$ is self-adjoint. Last, $\Dirac_n$ has compact resolvent because $|\Dirac_n|$ is the operator of multiplication by the length $\mathds{L}$, and $\mathds{L}$ is proper.
\end{proof}

We now want to prove our spectral triples are, in fact, metric. By \cite{Rieffel15b,LongWu21}, spectral triples defined over discrete group convolution C*-algebras using proper length function, as initially proposed by Connes \cite{Connes89}, provide examples of metric spectral triples when the length function has the bounded doubling property. Our spectral triples here are not of the same type, but we shall relate them.

We begin with the following lemma.

\begin{lemma}\label{doubling-lemma}
	If $l$ is a proper length function  on a discrete group $G$ with bounded doubling, then $\sqrt{l}$  is also a  proper length function with bounded doubling on $G$.
\end{lemma}

\begin{proof}
	For any $g \in G$, we first note that $\sqrt{l(g)} = 0$ if, and only if, $l(g) = 0$, i.e. if, and only if, $g$ is the unit of $G$.
	
	We then note that for all $g \in G$, we have $\sqrt{l(g^{-1})} = \sqrt{l(g)}$.
	
	Furthermore, let $g,g' \in G$. We then note that
	\begin{equation*}
	\sqrt{l(gg')}\leq \underbracket[1pt]{\sqrt{l(g)+l(g')}}_{l\text{ is a length function}} \leq \underbracket[1pt]{\sqrt{l(g)} + \sqrt{l(g')}}_{\forall x,y\geq 0 \; \sqrt{x+y}\leq\sqrt{x}+\sqrt{y}} \text.
	\end{equation*}
	
	Therefore, $\sqrt{l}$ is a length function on $G$.
	
	\medskip
	
	Now let $r \geq 1$. By definition,
	\begin{equation*}
	G_{\sqrt{l}}[r] = G_{l}[r^2] \text.
	\end{equation*}
	
	Therefore, $\sqrt{l}$ is proper. Moreover, since $l$ has the bounded doubling property, we observe that:
	\begin{equation*}
	|G_{\sqrt{l}}[2 r]| = |G_{l}[4 r^2]| \leq C\ |G_{l}[2r^2]| \leq C^2|G_{l}[r^2]| = C^2\ |G_{\sqrt{l}}[r]| \text.
	\end{equation*}
	This concludes the proof of our lemma.
\end{proof}

We then need a lemma relating the behavior of two spectral triples, one obtained by taking the absolute value of the Dirac operator of the other. In fact, such a result escapes us, but we can get close enough for our purpose, as follows.
\begin{lemma}\label{comp-lemma}
	Let $(\A,\Hilbert,\Dirac)$ be a spectral triple. If $r \in (0,1)$, then there exists $K_r \in (0,\infty)$ such that, for all $a \in \A$ such that $a\, \dom{\Dirac}\subseteq \dom{\Dirac}$ and $[\Dirac,a]$ is bounded, we have:
	\begin{equation*}
	\opnorm{[|\Dirac|^r,a]}{}{\Hilbert} \leq K_r \opnorm{[\Dirac,a]}{}{\Hilbert} \text.
	\end{equation*}
\end{lemma}

\begin{proof}
	The proof is identical to the proof of \cite[Lemma 6.23]{Brain15}
\end{proof}

We now can relate the spectral triples of the type defined here, and the spectral triples defined using length functions, and obtain a sufficient condition to prove certain spectral triples are indeed metric.

\begin{proposition}\label{metric-sp-cor}
	Let $\mathds{L}$ be a  proper length function with bounded doubling on a nilpotent discrete group $G$, let $\sigma$ be a $2$-cocycle of $G$ with values in $\T$, and let $E$ be a Hermitian space. Let $\lambda$ be the left regular $\sigma$-projective representation of $C^\ast(G,\sigma)$ on the Hilbert space $\ell^2(G)$. If $(C^\ast(G,\sigma),\ell^2(G)\otimes E,\Dirac)$ is a spectral triple such that $|\Dirac|$ is $M_{\mathds{L}}\otimes 1_E$, where $M_{\mathds{L}}$ is the multiplication operator by $\mathds{L}$, then $(C^\ast(G,\sigma),\ell^2(G)\otimes E, \Dirac)$ is a metric spectral triple.
\end{proposition}

\begin{proof}
	Let $\Hilbert\coloneqq\ell^2(G)\otimes E$. 
	By Lemma (\ref{comp-lemma}), there exists $K > 0$ such that, for all $a\in\dom{\Lip_\Dirac}$ with $\opnorm{[\Dirac,a]}{}{\Hilbert}$ bounded, we have:
	\begin{equation}\label{metric-lemma-eq1}
	\opnorm{[|\Dirac|^{\frac{1}{2}},a]}{}{\Hilbert} \leq\  K\ \opnorm{[\Dirac,a]}{}{\Hilbert} \text.
	\end{equation}
	
	Now,  a simple computation shows that
	\begin{equation*}
	|\Dirac|^{\frac{1}{2}} = M_{\sqrt{\mathds{L}}}\otimes 1_{E}\text.
	\end{equation*}
	Since $\sqrt{\mathds{L}}$ is a proper bounded doubling length function by Lemma (\ref{doubling-lemma}), by applying \cite{Rieffel15b,LongWu21} we get that  the spectral triple $(C^\ast(G,\sigma),\ell^2(G),M_{\sqrt{\mathds{L}}})$ is metric. It is then immediate that $(C^\ast(G,\sigma),\ell^2(G)\otimes E,|\Dirac|^{\frac{1}{2}})$ is metric as well.
	
	By Expression (\ref{metric-lemma-eq1}) and the Comparison Lemma \cite[Lemma 1.10]{Rieffel98a}, we conclude that $(C^\ast(G,\sigma),\ell^2(G)\otimes E,\Dirac)$ is metric.
\end{proof}

\begin{theorem}   \label{thm:triple is metric}
The spectral triple in Equation \eqref{Our-spectral-triple-eq} is metric.
\end{theorem}
%
%
%

\begin{proof}
	Let $\mathds{L} \coloneqq \sqrt{\sum_{j=1}^d x_j^2 + \mathds{F}^2}$. As proven in \cite[Corollary 4.17]{FarsiLatremolierePacker22}, the function $\mathds{L}$ is a proper length function which has bounded  doubling over $\left(\Z\left[\frac{1}{p}\right]\right)^d$. 
	
	A simple computation proves that, for all $\xi \in \ell^2\left(\Z\left[\frac{1}{p}\right]^d,E\right)$ and $g \in \Z\left[\frac{1}{p}\right]^d$,
	\begin{equation}\label{eq:abs value Dirac}
	|\Dirac| \xi (g) =  \left( \sum_{j=1} x_j(g)^2 +\mathds{F}(g)^2\right)^{\frac{1}{2}} \xi(g)=M_{\mathds{L}}\xi (g) \text.
	\end{equation}
	So by Proposition (\ref{metric-sp-cor}), the spectral triple $\left(C^\ast\left(\Z\left[\frac{1}{p}\right]^d,\sigma\right), \ell^2\left(\Z\left[\frac{1}{p}\right]^d,E\right), \Dirac\right)$ is metric.
\end{proof}

We also note the following fact for later use:

\begin{proposition}
	The spectral triple of Equation \eqref{Our-spectral-triple-eq} is finitely summable.
\end{proposition}

\begin{proof} By \cite[Proposition 1.2]{Rieffel15b} our length function $\mathds{L}$ is of polynomial growth since it has bounded doubling.
	Then one can apply \cite[Proposition 6]{Connes89} to end the proof, or do a direct calculation by using  Equation \eqref{eq:abs value Dirac}. \end{proof}


We are now ready to begin the  proof  our convergence result, 
Theorem \eqref{thm:sp  GH lim};  this theorem   will follow from  the Bridge Builder Theorem   \cite[Theorem 3.20]{FarsiLatremolierePacker22}. The first step in its proof will be the technical Lemma \eqref{le:LneqLinf} below.

For $n\in \N\cup\{\infty\}$, define
\begin{equation*}\label{eq:tech lemma Dn}
\dom{\Lip_n}\coloneqq \{ f \in \sa{C^*(G_n, \sigma)} : \    \Lip_n(f)\coloneqq \opnorm{[\Dirac_n,\lambda_n(f)]}{}{\ell^2(G_n, E) }  < \infty\}
\end{equation*}
and 
\begin{equation*}\label{eq:tech lemma defDn}
\forall f \in \dom{\Lip_n}: \quad \Lip_n(f)\coloneqq \opnorm{[\Dirac_n,\lambda_n(f)]}{}{\ell^2(G_n, E) }.
\end{equation*}

\begin{lemma}\label{le:LneqLinf} With notation as above, 
	for all $n \in \N$, if $f \in C_c(G_n)$, then
	\begin{equation*}
	\Lip_n(f) = \Lip_\infty(f) \text.
	\end{equation*}
	
\end{lemma}

\begin{proof} 
Recall that for all $g\in G_\infty$,  $ C^* (G_\infty,\sigma)$ act via the  left regular $\sigma$-projective representation $\lambda_\infty$ on $\ell^2(G_\infty,E)$ defined for all $\xi \in \ell^2(G_\infty,E)$ by:
	\begin{equation*}
	\lambda_\infty(g)\xi :  G_\infty \ni h  \longmapsto \sigma(g,h-g)\xi(h-g) \text.
	\end{equation*}
	 Now observe that, for $\xi \in \dom{\Dirac_\infty}$ and  $g,h \in G_\infty$ we have, if we define $\sigma_{g}^h \coloneqq \sigma(  g, h-g )$:

	Now fix $n\in \N$.    Let $C_n\subseteq G_\infty$ with $0 \in C_n$ (where $0$ is the zero in the group $G_\infty$), and such that the restriction of the canonical surjection from $G_\infty$ onto $\bigslant{G_\infty}{G_n}$ is a bijection --- i.e., every coset of $G_n$ is of the form $k + G_n$ for a unique $k \in C_n$; moreover if $k \in C_n$ and $k\neq 0$, then $k \notin G_n$.
	We then have by construction that
	\begin{equation*}
	\ell^2(G_\infty,E) = \overline{\oplus_{k \in C_n}} \ell^2(k+G_n,E) \text.
	\end{equation*}
	Moreover, for all $k \in C_n$, we also note that $\Dirac_\infty\ell^2(k+G_n,E) \subseteq \ell^2(k+G_n,E)$; in addition  for any  $f \in C_c(G_n)$ it also follows:  $\lambda_\infty(f) \ell^2(k+G_n,E) \subseteq\ell^2(k+G_n)$. Therefore, $[\Dirac_\infty,\lambda_\infty(f)] \ell^2(k+G_n,E) \subseteq \ell^2(k+G_n,E)$. Thus, if to  simplify (and with slight abuse of) the notation we write 
	$\opnorm{[\Dirac_\infty,\lambda_\infty(f)]}{}{\ell^2(k+G_n,E)}$ for $\opnorm{[\Dirac_\infty,\lambda_\infty(f)]_{|\ell^2(k+G_n,E)}}{}{\ell^2(k+G_n,E)}$ we have:
	\begin{equation}\label{eqdef:defofsup}
	\opnorm{[\Dirac_\infty,\lambda_\infty(f)]}{}{\ell^2(G_\infty,E)} = \sup_{k \in C_n} \opnorm{[\Dirac_\infty,\lambda_\infty(f)]}{}{\ell^2(k+G_n,E)} \text.
	\end{equation}
	
	By the definitions, $\opnorm{[\Dirac_\infty,\lambda_\infty(f)]_{|\ell^2(G_n,E)}}{}{\ell^2(G_\infty,E)} = \opnorm{[\Dirac_n,\lambda_n(f)]}{}{\ell^2(G_n,E)}$.
	
  Now, fix $k\in C_n$, $k\not= 0$, and let $\xi \in \ell^2(G_n + k,E)$ with $\norm{\xi}{\ell^2(G_n+k,E)}\leq 1$.  
  
  Now, note that for  each $h \in G_n$:
	\begin{align*}
	[\Dirac_\infty,\lambda_\infty(f)]\xi(k + h) &= \sum_{g \in G_\infty}\left( \sum_{j=1}^d x_j(g)\gamma_j \right)\sigma_g^{k+h} f(g)\xi(k + h-g)\\
	&+ \sum_{g \in G_\infty}\left(\mathds{F}(k + h)-\mathds{F}(k + h-g))\gamma_{d+1}\right)\sigma_g^{k+h} f(g)\xi(k + h-g) \text.
	\end{align*}
	
	Now, $k \notin G_n$ by construction. Let $m = \min\{ r \in \N : k \in G_r \}$. Then $k \in G_r$ and $r>n$. For any $w\in G_n$, we note that $k + w \in G_m$ since $G_n \subseteq G_m$. Moreover, if $k+w \in G_p$ for some $p \in \{n,\ldots,m-1\}$, then $k \in G_p - w \subseteq G_p$, and we get a contradiction with the minimality of $r$. So $\min\{r\in\N: k+w \in G_r\} = m$. Consequently:
	\begin{equation*}
		\mathds{F}(k + h - g) = \mathds{F}(k + h) \text.
	\end{equation*}
	we thus get
	\begin{align*}
	\quad&\norm{[{\Dirac}_\infty,\lambda_\infty(f)]\xi(k + h)}{E}^2 \\
	&=\sum_{g \in G_n} |f(g)|^2 \norm{(\sum x_j(g)\gamma_j +(\mathds{F}(k + h)-\mathds{F}(k + h-g)) \gamma_{d+1})\xi(k+h-g)}{E}^2 \\
	&=\sum_{g \in G_n} |f(g)|^2 \left(\sum |x_j(g)|^2 + (\mathds{F}(k + h)-\mathds{F}(k + h-g))^2 \right) \norm{\xi(k+h-g)}{E}^2 \\
	&=\sum_{g \in G_n} |f(g)|^2 \left(\sum |x_j(g)|^2 + ( \mathds{F}(h)- \mathds{F}(h-g) )^2 \right) \norm{\xi(k+h-g)}{E}^2 \\
	&= \sum_{g \in G_n} |f(g)|^2\norm{\left(\sum x_j(g)\gamma_j \right)\xi(k+h-g)}{E}^2\\
	&=\norm{[\Dirac_n,\lambda_n(f)]\xi(k+h)}{E}^2\text.
	\end{align*}
	
	Now
	\begin{equation}\label{eq-eq}
	\begin{split}
	\norm{[{\Dirac}_\infty,\lambda_\infty(f)]\xi}{\ell^2(G_n+k,E)}^2 
	&=\sum_{h \in G_n} \norm{[{\Dirac}_\infty,\lambda_\infty(f)]\xi(k+h)}{E}^2 \\
	&= \sum_{h\in G_n} \norm{[\Dirac_n,\lambda_n(f)]\xi(k+h)}{E}^2 \\
	&\leq \opnorm{[\Dirac_n,\lambda_n(f)}{2}{\ell^2(G_n,E)} \norm{\xi}{\ell^2(G_n+k,E)}^2 \\
	&= \opnorm{[\Dirac_n,\lambda_n(f)}{2}{\ell^2(G_n,E)}\text.
	\end{split}
	\end{equation}
	
	Hence $\opnorm{[{\Dirac}_{\infty},\lambda_\infty(f)]}{}{\ell^2(G_n+k,E)} \leq \opnorm{[\Dirac_n,\lambda_n(f)}{}{\ell^2(G_n,E)}$. By Equations \eqref{eqdef:defofsup} and \eqref{eq-eq}, we conclude
	\begin{equation*}
	\opnorm{[{\Dirac}_{\infty},\lambda_\infty(f)]}{}{\ell^2(G_\infty,E)} = \opnorm{[{\Dirac}_n,\lambda_n(f)]}{}{\ell^2(G_n,E)}\text,
	\end{equation*}
	as needed.
\end{proof}

With the lemma above, we are now in a position to  prove our convergence result. The key step in the proof of Theorem 
\eqref{thm:sp  GH lim} is an application of the Bridge Builder Theorem  \cite[Theorem 3.20]{FarsiLatremolierePacker22}, with the bridge builder $\pi$ equal to the identity of $C^\ast(G_\infty,\sigma)$ (the identity  is obviously a quantum isometry). For the convenience of the reader,  we recall that the identity is a bridge builder when, for all $\varepsilon > 0$, there exists $N\in\N$ such that if $n\geq N$, then
\begin{equation}\label{eq:bb-1}
\forall a \in \dom{\Lip_\infty} \quad \exists b \in \dom{\Lip_n}: \quad \Lip_n(b) \leq \Lip_\infty(a) \text{ and }\norm{a - b}{C^\ast(G_\infty,\sigma)} < \varepsilon\Lip_\infty(a) \\
\end{equation}
and
\begin{equation}\label{eq:bb-2}
\forall b \in \dom{\Lip_n} \quad \exists a \in \dom{\Lip_\infty}: \quad \Lip_\infty(a)\leq\Lip_n(b) \text{ and }\norm{a-b}{C^\ast(G_\infty,\sigma)}<\varepsilon\Lip_n(b) \text.
\end{equation}

We thus establish the following.
\begin{theorem} \label{thm:sp  GH lim}
	Using the notation of Equations \eqref{Our-spectral-triple-eq} and \eqref{Our-spectral-triple-n-eq}), we obtain
	\begin{equation*}
		\lim_{n\rightarrow\infty} \spectralpropinquity{}((C^\ast(G_\infty,\sigma),\ell^2(G_\infty,E),\Dirac_\infty),(C^\ast(G_n,\sigma),\ell^2(G_n,E),\Dirac_n)) = 0\text,
	\end{equation*}
	where $\spectralpropinquity{}$ denotes the spectral propinquity on metric spectral triples of \cite{Latremoliere21a, Latremoliere18g, Latremoliere22}.
\end{theorem}

\begin{proof} 
	 We start by 
	letting  $\widehat{G_n}$ denote the Pontryagin dual of $G_n$ (we will use the multiplicative notation for $\widehat{G_n}$). The dual action $\beta$ of $\widehat{G_n}$ on $C^\ast(G_n,\sigma)$ is unitarily implemented by defining, for each $z\in \widehat{G_n}$, the unitary $v^z$ of $\ell^2(G_n,E)$ which is given by,  for all $\xi \in \ell^2(G_n)\otimes E$:
	\begin{equation*}
	\quad v^z \xi : g \in G_n \longmapsto \overline{z(g)} \xi(g) ( \, = z(-g)\xi(g) )\text.
	\end{equation*}
	
	We then note that:
	\begin{equation*}
	\forall z \in \widehat{G_n} \quad  v^z \lambda_E(g) (v^z)^\ast = \beta^z \lambda_E(g) \text.
	\end{equation*}
	
	By construction, $\Dirac_n$ commutes with $v^z$ for all $z \in \widehat{G_n}$, so $\beta$ acts by full quantum isometries on the quantum compact metric space  $(C^\ast(G_n,\sigma),\Lip_n)$, for all $n$.
	
Let $\mu$ be the Haar probability measure on the Pontryagin dual $\widehat{G_n}$. As seen in \cite[Lemma 3.1]{Latremoliere05} and  \cite[Theorem 8.2]{Rieffel00}, there exists a sequence $(\varphi_k)_{k\in\N}$ of non-negative functions on $\widehat{G_n}$, each obtained as a finite linear combination of characters of $\widehat{G_n}$ (i.e. of the form $z\in \widehat{G_n} \mapsto z(g)$ for some $g \in G_n$, by Pontryagin duality), such that $\int_{\widehat{G_n}}\varphi_k \,d\mu = 1$ for all $k\in \N$, and $(\varphi_k)_{k\in\N}$ converges, in the sense of distributions, to the Dirac measure at $1 \in \widehat{G_n}$, i.e., for all $f \in C(\widehat{G_n})$,
\begin{equation*}
\lim_{k\rightarrow\infty} \int_{\widehat{G_n}} f(z) \, \varphi_k(z)d\mu(z) = f(1) \text.
\end{equation*}
We define, for each $k,n \in \N$, the continuous linear endomorphism:
\begin{equation*}
\mathds{E}_k: C^*(G_n,\sigma)\to C^*(G_n,\sigma),\quad \mathds{E}_k: a \mapsto \int_{\widehat{G_n}} \beta^z(a) \, \varphi_k(z)d\mu(z) \text.
\end{equation*}
It it shown in \cite[Lemma 3.1]{Latremoliere05} and  \cite[Theorem 8.2]{Rieffel00} that the  range of $\mathds{E}_k$ is a finite dimensional subspace of $C^*(G_n,\sigma)$, and so contained in $C_c(G_n,\sigma)$.  Moreover, since the dual action acts by quantum isometries, and for all $n$, $\Lip_n$ is lower semi-continuous, it follows that  
given $a\in \dom{\Lip_n}$, and $\varepsilon > 0$, there exists $K \in \N$ such that if $k\geq K$, then $\norm{a-\mathds{E}_{k}(a))}{C^\ast(G_n,\sigma)} < \varepsilon$ and 
$\Lip_{n}(\mathds{E}_{k}(a))\leq\Lip_n(a)$.  
Moreover (cf. \cite[Equation 4.10]{FarsiLatremolierePacker22}):
\begin{equation}\label{fejer-eq}
\left\{ a\in\dom{\Lip_\infty}: \Lip_\infty(a) \leq 1 \right\} = \overline{\left\{ b \in \dom{\Lip_\infty}\cap C_c(G_n,\sigma) : \Lip_\infty(b) \leq 1 \right\}} \text.
\end{equation}

Now let $\varepsilon > 0$ and fix $\mu \in \StateSpace(C^\ast(G_\infty,\sigma))$. Since $(C^\ast(G_\infty,\sigma),\ell^2(G_\infty,E),\Dirac_\infty)$ is a metric spectral triple, there exists a finite set $F\subseteq \dom{\Lip_\infty}$ such that $\Lip_\infty(b) \leq 1$ and $\mu(b) = 0$ for all $b\in F$, and if $a\in\dom{\Lip_\infty}$ with $\Lip_\infty(a)\leq 1$ and $\mu(a) = 0$, then there exists $b \in F$ such that $\norm{a-b}{C^\ast(G_\infty,\sigma)} < \frac{\varepsilon}{3}$. Moreover, by Equation \eqref{fejer-eq},
without loss of generality we can assume $b\in C_c(G_n)$, for all $n\geq N'$, for some $N'\in \N$. 
	
Since $F$ is finite, there exists $N\in\N$ such that, for all $b \in F$ and $n>N$, we have 

\[\norm{b-\mathds{E}_N(b)}{C^\ast(G_\infty,\sigma)}=\norm{b-\mathds{E}_N(b)}{C^\ast(G_n,\sigma)} < \frac{\varepsilon}{3}.\] 

Thus, if $a \in \dom{\Lip_\infty}$ 
with $\Lip_\infty(a)\leq 1$ and $\mu(a) = 0$, then choosing $b \in F$ with $\norm{a-b}{C^\ast(G_\infty,\sigma)} < \frac{\varepsilon}{3}$, we conclude that
\begin{align*}
\norm{a-\mathds{E}_N(a)}{C^\ast(G_\infty,\sigma)} &\leq \norm{a-b}{C^\ast(G_\infty,\sigma)} + \norm{b-\mathds{E}_N(b)}{C^\ast(G_\infty,\sigma)} +\\ &+\norm{\mathds{E}_N(b-a)}{C^\ast(G_\infty,\sigma)} < \varepsilon \text.
\end{align*}

Moreover, by Lemma \eqref{le:LneqLinf}, we also note that $\Lip_n(\mathds{E}_N(a)) = \Lip_\infty(\mathds{E}_N(a)) \leq \Lip_\infty(a)$.

Conversely, let $b \in \dom{\Lip_n}$. Then $\Lip_\infty(b) = \Lip_n(b)$ and of course, $\norm{b-b}{} = 0 < \varepsilon$.

Therefore we have checked Equations 
\eqref{eq:bb-1} and \eqref{eq:bb-2} and so  the identity  is a bridge builder. We  conclude that our theorem holds by \cite[Theorem 3.20]{FarsiLatremolierePacker22}.
\end{proof}

%
%

\section{$K$-homology for noncommutative $2$-solenoids and metric spectral triples}
\label{sec:K-hom}

We compute here the $K$-homology for noncommutative $2$-solenoids following \cite{Latremoliere11c}. However we note that, since these algebras arise as inductive limits of $C^*-$algebras, their $K$-homology is not very rich; this  holds even more markedly in the case of Bunce-Deddens algebras, which have their  $K^0$-groups  equal to zero  \cite{Klimek23}.
See also   \cite{weibel} Exercise 3.5.3 on page 85 and \cite{Latremoliere11c}.

We start by noting that we can apply the UCT to noncommutative solenoids because they are inductive limits of nuclear $C^*$-algebras, hence nuclear. 

By applying the universal coefficient theorem  to the $C^*$-algebras $A,\C$, we get  the short exact sequences of  \cite{Rosenberg-Schochet} \cite[page 9]{dadarlat}, in accordance with \cite[Corollary 1.19]{Rosenberg-Schochet}, \cite[Exercise 23.15.1]{blackadar} and also in agreement with \cite[Remark 4.6]{SchochetII}:

\begin{equation}\label{eq:odd case KK}
0\to Ext^1_\Z(K_0(A), \Z) \to 
KK_1(A,\C) \to   Hom(K_1(A),  \Z)\to 0
\end{equation}

and

\begin{equation}\label{eq:even case KK}
0\to   Ext^1_\Z(K_1(A), \Z) \to 
KK_0(A,\C) \to Hom((K_0(A), \Z))\to 0.
\end{equation}

This latter formula is also corroborated by \cite[Proposition 5.6]{Klimek23}.
We recall that the  above short exact sequences  split (unnaturally).  Moreover by \cite[Equation (3.3)]{Latremoliere11c}:

\[
K_0(\nctwosolenoid)\hbox{ is an extension of $\Z[1/p]$ by $\Z$};  \quad K_1(\nctwosolenoid) \cong \Z[1/p]^2,
\]

that is,

\begin{equation}\label{eq:key ext}
0 \to \Z \xrightarrow{\iota} K_0(\nctwosolenoid) \to \Z[1/p] \to 0.
\end{equation}

We have, see  \cite[Remark 3.17, Page 183]{Latremoliere11c} or  Weibel's book \cite{weibel} Exercise 3.5.3 on page 85:

\[
Ext(\Z[1/p], \Z ) \cong \faktor{I_p}{\Z},
\]
where $I_p$ is the group of $p$-adic integers.
(From now on $Ext^1$ will be just called $Ext$ for simplicity.)

We can now apply $Hom(., \Z) $ to the short exact sequence of Equation \eqref{eq:key ext}, to get

\begin{equation}
\begin{split}
0\to Hom(\Z[1/p], \Z ) \to Hom(K_0(\nctwosolenoid), \Z ) \to Hom(\Z, \Z )\cong \Z \to\\
Ext(\Z[1/p], \Z ) \to Ext(K_0(\nctwosolenoid), \Z ) \to 0.
\end{split}
\end{equation}

Since $Hom(\Z[1/p], \Z ) \cong 0$, this short exact sequence simplifies to:

\begin{equation}
\begin{split}
0\to  Hom(K_0(\nctwosolenoid), \Z ) \to Hom(\Z, \Z )\cong \Z \xrightarrow{i}\\
Ext(\Z[1/p], \Z ) \to Ext(K_0(\nctwosolenoid), \Z ) \to 0.
\end{split}
\end{equation}

It follows that

\[
Ext(K_0(\nctwosolenoid), \Z ) \cong Ext(\Z[1/p], \Z )/ i(\Z) \cong \big( I_p /\Z \big)/ i(\Z) .
\]

In conclusion, from Equations \eqref{eq:odd case KK} and \eqref{eq:even case KK} together with the above calculations we get

\begin{proposition} 
The $K$-homology groups of the noncommutative solenoids are given by:
\begin{equation}\label{eq:odd case KK fin}
K^1(\nctwosolenoid) \cong  (I_p /\Z )/ i(\Z) .
\end{equation}

and

\begin{equation}\label{eq:even case KK fin}
K^0(\nctwosolenoid) \cong   \left({ I_p}/{\Z}\right)^2 \oplus 
  Hom((K_0(\nctwosolenoid), \Z)).
\end{equation}
	
	\end{proposition}

We will now use our prior results to prove that 
no even, metric, and summable spectral triple on   noncommutative $2$-solenoids $\nctwosolenoid$ can represent a non-trivial $K^0$-homology class detected by using the index map.

For, say we have such a metric spectral triple on the noncommutative $2$-solenoid. Then, if we denote by $D$ the spectral triple's Dirac operator, we have:

\[
[D,1] = 0,
\]

because  for metric spectral triples  we must have $L_D(1) = 0$, as  $L_D(a)$ is the norm of $[D,a]$, for all $a$ in the smooth subalgebra of the spectral triple  $C^*$-algebra. So, if we denote by $\Psi_D$ the  morphism  from 

\[
\Psi_D:K_0(\nctwosolenoid)  \to \Z
\]
induced by the pairing of  the $K$-homology class of  $D$  with a vector $C^*$-bundle bundle (this can be for example represented in KK-theory by the standard pairing of $K$-homology with $K$-theory), using the notation of Equation \eqref{eq:key ext}, we conclude that $\iota(\Z)$ must be  in the kernel of $\Psi_D$  (see the local even index  formula \cite{Connes} for a quick way to prove this; in particular see the commutator term $[\Dirac, p]$ which must be zero in this case). Thus, the index map $\Psi_D$ factors in the extension of Equation \eqref{eq:key ext}, i.e. $\Psi_D$ induces a group morphism from $\Z[1/p]$ to $\Z$. But there is only one such morphism: the 0 one.
Therefore, the index map for any even metric spectral triple over noncommutative solenoids must be zero. 

So  all constructions of even metric spectral triples on noncommutative solenoids we can make will have index map zero, making it very difficult to know whether they are actually trivial or not, as the only possible nonzero component  comes from the   $Ext$ group. 

For odd spectral triples the situation is similar, and we cannot even hope to use the local index theorem to show the nontriviality of the Dirac class since the $Hom$ term in the $K^1$-homology vanishes.

Therefore

\begin{proposition}\label{prp:no-K-hom-class} Any finitely summable  even or odd metric spectral triple on a noncommutative $2$-solenoid  has zero index map, and so  the only possible nonzero component of the  $K$-homology class associated to the Dirac operator  comes from the   $Ext$ groups.
\end{proposition}

\providecommand{\href}[2]{#2}

\vfill


\begin{thebibliography}{10}
	
		
%
	
	
%
%
%
	
	
	
%
	
		\bibitem{Baggett}
	{L}. {Baggett}, \emph{Functional Analysis: A Primer},  CRC Press, 1991.
	
	\bibitem{Brain15} S. Brain, B.Mesland, and  W. D. van Suijlekom
	\emph{Gauge theory for spectral triples and the unbounded Kasparov product},
	J. Noncommut. Geom.10 (2016),  135–206.
	
	
	\bibitem{BCFL}  J.  Bassi, R. Conti, C. Farsi, and  F. Latr\'emoli\`ere, \emph{ Isometry groups of inductive limits of metric spectral triples and Gromov-Hausdorff convergence}, J.  London Math. Soc.,  108 (2023),   1488--1530,
	arXiv:2302.09117.
	
	\bibitem{blackadar} B. blackadar, \emph{K-theory for operator algebras} : MSRI publications, 5. Cambridge University Press (1998).
	
	
		\bibitem{Bratteli86}
	{O}. {B}ratteli, \emph{Derivations, dissipations and group actions on
		{$C^\ast$}-algebras}, Lecture notes in Mathematics, no. 1229,
	Springer-Verlag, Berlin, 1986.
	
%
	
%
	

\bibitem{Connes81} A. Connes, \emph{{$C^\ast$}-algebras et g{\'e}om{\'e}trie differentielle}, C. R. Acad. Sci. Paris Sér. A-B, \textbf{290} (1981), no. 13, pp. A599–A604

\bibitem{Connes89}	\bysame,  \emph{Compact metric spaces, Fredholm modules, and hyperfiniteness}, 
Ergodic Theory Dynam. Systems \textbf{9} (1989), no. 2, 207--220.
	
	\bibitem{Connes91}	\bysame, 
	\emph{Caract\`eres des repr\'esentations $\theta$-sommables des groupes discrets}, 
	C. R. Acad. Sci. Paris Sér. I Math. \textbf{312} (1991), no.9, 661--666.
	
	
	\bibitem{Rieffel15b}
	M. Christ and {M.}~{A.} {R}ieffel, \emph{Nilpotent group
		{$C^\ast$-algebras}-algebras as compact quantum metric spaces}, Canad. Math.
	Bull. \textbf{60} (2017),  77--94, arXiv: 1508.00980.
	
	
	
		
	\bibitem{Cipriani96}
	{C}ipriani, {D}.~{G}uido, and {S}.~{S}carlatti {F}, \emph{A remark on trace
		properties of {$K$}-cycles}, J. Operator Theory \textbf{35} (1996), no.~1,
	179--189.
	
	\bibitem{Connes89}
	A.~{C}onnes, \emph{Compact metric spaces, Fredholm modules and
		hyperfiniteness}, Ergodic Theory Dynam. Systems \textbf{9} (1989), 
	207--220.
	
	\bibitem{Connes}
	\bysame, \emph{Noncommutative geometry}, Academic Press, San Diego, 1994.
	
	
	\bibitem{CvS}	A. Connes and W. D. van Suijlekom, \emph{Spectral truncations in noncommutative geometry
		and operator systems}, Comm. Math. Phys. 383 (2021), 2021–2067.
	
	\bibitem{dadarlat} {\sc Dadarlat} \emph{Notes from the Copenhagen conference}, in Applications  of the UCT for  C*-algebras, Notes by Dimitris Gerontogiannis.
	
	
		\bibitem{Dietmar}
	A.~{D}ietmar, \emph{A First Course in Harmonic Analysis}, Second Edition,  Springer, 2005. 
	

	
	
%
	
	\bibitem{FarsiPacker22}
	C.~Farsi, T.~Basa~Landry, N.~S.~Larsen, and J.~Packer, \emph{Spectral triples  for noncommutative solenoids and a Wiener's lemma}, J. Noncommutative Geom, to appear (2024), ArXiv:  2212.07470.
	
	
		\bibitem{FarsiLatremolierePacker22}
	C.~Farsi, F. Latr\'emoli\`ere, and J.~Packer, \emph{Convergence of inductive sequences of spectral triples for
the spectral propinquity}, Adv. Math, Paper No. 109442, 59 pp.
	
	
	\bibitem{Floricel19}
	R.~Floricel and A.~Ghorbanpour, \emph{On inductive limit spectral triples},
	Proc. Amer. Math. Soc. \textbf{147} (2019),  3611--3619, ArXiv:
	1712.09621.
	
	
	
	
	\bibitem{Glazier2020}	L. Glaser and A. B. Stern, \emph{Reconstructing manifolds from truncations of spectral triples}, J. Geom. Phys. 159
	(2021), 103921, 17. 
	
		\bibitem{GraciaVarilly}
	J.~M. {G}racia {B}ondia, J.~C. {V}arilly, and H.~{F}igueroa, \emph{Elements of
		noncommutative geometry}, Birkh\"auser, 2001.
	

\bibitem{Gromov81}
M.~{G}romov, \emph{Groups of polynomial growth and expanding maps}, Publ. Math.
Inst. Hautes {\'E}tudes Sci. \textbf{53} (1981), 53--78.

\bibitem{Gromov}
\bysame, \emph{Metric structures for Riemannian and non-{R}iemannian spaces}, Progress in Mathematics, Birkh\"a user, 1999.
	
		\bibitem{Hadfield04}	Hadfield, Tom 
	K-homology of the rotation algebras $A_\theta$. 
	Canad. J. Math. 56 (2004), no. 5, 926--944. 
	
	
	
		\bibitem{Klimek23} S. Klimek, M. McBride, and J.W.  Peoples,
	\emph{Aspects of noncommutative geometry of Bunce-Deddens algebras}, 
	J. Noncommut. Geom.17 (2023),  1391--1423.
	
	
	
	
%
	
%
	
	
	\bibitem{Latremoliere20a}
	{T}. {Landry}, {M}. {L}apidus, and {F}. {L}atr{\'e}moli{\`e}re, \emph{Metric
		approximations of the spectral triple on the Sierpinki gasket and other
		fractals}, Adv. Math. \textbf{385} (2021), Paper No. 107771, 43 pp.
	
	
%
	
	
	\bibitem{Latremoliere05}
	{F}. {L}atr{\'e}moli{\`e}re, \emph{Approximation of the quantum tori by finite
		quantum tori for the quantum {G}romov-{H}ausdorff distance}, J. Functional Analalysis
	\textbf{223} (2005), 365--395, math.OA/0310214. 
	
		\bibitem{Latremoliere15c}
	\bysame, \emph{Curved noncommutative tori as {L}eibniz compact quantum metric
		spaces}, J. Math. Phys. \textbf{56} (2015),  123503, 16 pages, arXiv:
	1507.08771.
	
	\bibitem{Latremoliere13b}
	\bysame, \emph{The dual {G}romov-{H}ausdorff propinquity}, J. Math. Pures
	Appl. \textbf{103} (2015),  303--351, arXiv: 1311.0104.
	
	\bibitem{Latremoliere13c}
	\bysame, \emph{Convergence of fuzzy tori and quantum tori for the quantum
		Gromov-Hausdorff propinquity: an explicit approach.}, M\"unster J. Math.
	\textbf{8} (2015),  57--98, arXiv: math/1312.0069.
	
	\bibitem{Latremoliere15b}
	\bysame, \emph{Quantum metric spaces and the {G}romov-{H}ausdorff propinquity},
	Noncommutative geometry and optimal transport, Contemp. Math., \textbf{676}, Amer.
	Math. Soc., 2015, arXiv: 150604341, 47--133.
	
	
\bibitem{Latremoliere16b}
	\bysame, \emph{Equivalence of quantum metrics with a common domain}, J. Math.
	Anal. Appl. \textbf{443} (2016), 1179--1195, arXiv: 1604.00755.
	
	\bibitem{Latremoliere13}
	\bysame, \emph{The quantum {G}romov-{H}ausdorff propinquity}, Trans. Amer.
	Math. Soc. \textbf{368} (2016),  365--411.
	
		\bibitem{Latremoliere15}
	\bysame, \emph{A compactness theorem for the dual {G}romov-{H}ausdorff
		propinquity}, Indiana Univ. Math. J. \textbf{66} (2017),  1707--1753,
	arXiv: 1501.06121.
	
	
\bibitem{Latremoliere14}
	\bysame, \emph{The triangle inequality and the dual {G}romov-{H}ausdorff
		propinquity}, Indiana Univ. Math. J. \textbf{66} (2017),  297--313,
	arXiv: 1404.6633.
	
	
	\bibitem{Latremoliere16c}
	\bysame, \emph{The modular Gromov-Hausdorff propinquity}, Dissertationes
	Math. \textbf{544} (2019), 1--70, arXiv: 1608.04881.
	
	
	
\bibitem{Latremoliere18c}
	\bysame, \emph{Convergence of {C}auchy sequences for the covariant
		Gromov-Hausdorff propinquity}, J. Math. Anal. Appl. \textbf{469} (2019), 378--404, arXiv: 1806.04721.
	
		\bibitem{Latremoliere18a}
	\bysame, \emph{Convergence of {H}eisenberg modules for the modular
		Gromov-Hausdorff propinquity}, J. Operator Theory \textbf{84} (2020),
	211--237.
	
	\bibitem{Latremoliere18b}
	\bysame, \emph{The covariant {G}romov-{H}ausdorff propinquity}, Studia Math.
	\textbf{251} (2020),  135--169, arXiv: 1805.11229.
	
	
	\bibitem{Latremoliere21a}
	\bysame, \emph{Convergence of spectral triples on fuzzy tori to spectral
		triples on quantum tori}, Comm. Math. Phys. \textbf{388} (2021), 
	1049--1128, arXiv: 2102.03729.
	
	\bibitem{Latremoliere18d}
	\bysame, \emph{The dual-modular {G}romov-{H}ausdorff propinquity and
		completeness}, {J}. {N}oncomm. {G}eometry \textbf{115} (2021), 
	347--398.
	
	
	\bibitem{Latremoliere18g}
	\bysame, \emph{The {G}romov-{H}ausdorff propinquity for metric spectral triples}, Adv. Math. \textbf{404} (2022), Paper No. 108393, 56pp.
	
	
	\bibitem{Latremoliere22}
	\bysame,  \emph{Continuity of the spectrum of Dirac
		operators of spectral triples for the spectral propinquity}, Math. Ann.  Pub Date: 2023--07--06 , DOI:10.1007/s00208-023-02659-x.
	
	
	\bibitem{Latremoliere16}
	{F}. {L}atr{\'e}moli{\`e}re and {J}. {P}acker, \emph{Noncommutative solenoids and the
		{G}romov-{H}ausdorff propinquity}, Proc. Amer. Math. Soc. \textbf{145}
	(2017),  1179--1195, arXiv: 1601.02707.
	
	
	\bibitem{Latremoliere11c}
	\bysame, \emph{Noncommutative solenoids}, New York J. Math. \textbf{24A}
	(2018), 155--191, arXiv: 1110.6227.
	
	
	\bibitem{Leimbach24} 	M. Leimbach, and W. D. van Suijlekom, \emph{Gromov-Hausdorff convergence of spectral truncations for tori}, 	Adv. Math. \textbf{439} (2024), Paper No. 109496, 26 pp..
	
	
	
		\bibitem{LongWu21}
	B.~Long and W.~Wu, \emph{Twisted bounded-dilation group {$C^\ast$}-algebras as
		{$C^\ast$}-metric algebras}, Sci. China Math. \textbf{64} (2021), no.~3,
	547--572.
	

		\bibitem{Long-Wu1}  B. Long and W. Wu,  Twisted group $C^*$-algebras as compact
	quantum metric spaces, {\it Results in Mathematics}, 2016,
	published online June 8, 2016
	DOI 10.1007/s00025-016-0562-7.
	
	

	
	\bibitem{Ozawa05}
	{N}. {O}zawa and M.~A. {R}ieffel, \emph{Hyperbolic group {$C\sp\ast$}-algebras
		and free product {$C\sp\ast$}-algebras as compact quantum metric spaces},
	Canad. J. of Math. \textbf{57} (2005), 1056--1079, arXiv: math/0302310.
	
	
		\bibitem{Pinsky}
	M. Pinsky, \emph{Introduction to Fourier analysis and wavelets}, Graduate Studies in Mathematics, vol. 102, American Mathematical Society, 2002.
	
	
	
	
	\bibitem{Rieffel98a}
	M.~A. {R}ieffel, \emph{Metrics on states from actions of compact groups}, Doc.
	Math. \textbf{3} (1998), 215--229, math.OA/9807084.
	
	\bibitem{Rieffel99}
	\bysame, \emph{Metrics on state spaces}, Doc. Math. \textbf{4} (1999),
	559--600, math.OA/9906151.

	
	\bibitem{Rieffel00}
	\bysame, \emph{{G}romov-{H}ausdorff distance for quantum metric spaces}, Mem.
	Amer. Math. Soc. \textbf{168} (2004),  1--65, math.OA/0011063.
	
		\bibitem{Rieffel15}
	Marc~A. Rieffel, \emph{Matricial bridges for \lq \lq matrix algebras converge to the
		sphere"}, Operator algebras and their applications, Contemp. Math., vol.
	\textbf{671}, Amer. Math. Soc., Providence, RI, 2016, arXiv: 1502.00329, p.~209--233.
	
	
	
	
	
	\bibitem{Rieffel21} \bysame,  \emph{Dirac operators for matrix algebras converging to coadjoint orbits}, Comm. Math. Physics, to appear.
	
	
		\bibitem{Rosenberg-Schochet}	J. Rosenberg and C. Schochet, \emph{The K\"unneth theorem and the universal coefficient
		theorem for Kasparov’s generalized K-functor},  Duke Math. J. 55 (1987), no. 2, 431--474.
	
		\bibitem{SchochetII}C. L. Schochet, 	\emph{The Fine Structure of the Kasparov Groups II:
		Topologizing the UCT},  Journal of Functional Analysis 194, 263--287.

	
	

	
		
	\bibitem{weibel} C. Weibel, \emph{An introduction to homological algebra} : Cambridge University Press (1994).
	

	

\end{thebibliography}
\end{document}